\documentclass[psamsfonts, amssymb]{amsart}
\usepackage{dsfont}
\numberwithin{equation}{section}
\theoremstyle{plain}
\newtheorem{thm}{Theorem}
\newtheorem{lem}[thm]{Lemma}

\theoremstyle{remark}
\newtheorem{remark}[thm]{Remark}

\newcommand{\drm}{{\mathrm d}}

\newcommand{\be}{\begin{equation}}
\newcommand{\ee}{\end{equation}}
\newcommand{\irm}{{\mathrm i}}
\newcommand{\Lp}{{L_\mathrm{per}}}

\newcommand{\manu}{\cite{BLM}}
\newcommand{\Zz}{\mathbb{Z}}

\title[Non-self-adjoint problems with real spectrum]{On a class of non-self-adjoint periodic boundary value problems with discrete real spectrum}

\author{Lyonell Boulton}
\address{L. B.: Department of Mathematics, Heriot-Watt University\\
and Maxwell Institute for Mathematical Sciences\\
Riccarton, Edinburgh EH14 4AS, United Kingdom}
\email{L.Boulton@hw.ac.uk}
\author{Michael Levitin}
\address{M. L.: Department of Mathematics, University of Reading\\
Whiteknights, PO Box 220, Reading RG6 6AX, United Kingdom}
\email {M.Levitin@reading.ac.uk}
\author{Marco Marletta}
\address{M. M.: Cardiff School of Mathematics, Cardiff University\\
and Wales Institute of Mathematical and Computational Sciences\\ 
Senghennydd Road, Cardiff CF24 4AG, United Kingdom}
\email {Marco.Marletta@cs.cardiff.ac.uk}
\thanks{The research of Michael Levitin was partially supported by the EPSRC grant EP/D054621.}

\date{March 2010}

\begin{document}

\maketitle

\section{Introduction}

We study the operator $L_\mathrm{per}$ defined by
\begin{equation}\label{e.0}
\Lp u := i \epsilon (f(x)u'(x))'+iu'(x) 
\end{equation}
in which $f$ is a given $2\pi$-periodic function having the following properties:
\begin{equation}\label{eq:f_sym}
f(x+\pi) = -f(x)\,,\qquad f(-x) = -f(x)\,; 
\end{equation}
and also 
\begin{equation}\label{eq:f_pos}
f(x)>0\qquad\text{for }x\in (0,\pi)\,. 
\end{equation}
In particular
it follows that $f(\pi\Zz)=0$. We assume that $f$ is continuous, and differentiable 
except possibly at a finite number of points, the points of non-differentiability 
excluding $\pi\Zz$. We assume that $f'(0) = 2/\pi$ and that $0<\epsilon<\pi$. 

 We consider \eqref{e.0} on the domain 
\begin{equation}\label{eq:m2} 
  {\mathcal D} = \{ u \in L^2(-\pi,\pi) \, | \, \Lp u \in L^2(-\pi,\pi); \;\; u(-\pi) = u(\pi) \} . 
\end{equation}

\begin{remark} Of course it is not obvious that functions $u\in L^2(-\pi,\pi)$ such that $\Lp u\in L^2(-\pi,\pi)$ have
boundary values $u(\pm \pi)$; this was proved in \manu, where we showed that if $u\in L^2(-\pi,\pi)$
and $\Lp u\in L^2(-\pi,\pi)$ then $u\in H^1(-\pi,\pi)$.
\end{remark}

The main results of this paper are
\begin{thm}\label{thm:1}
The spectrum of operator $\Lp$ is 
\begin{itemize}
\item[(a)] real
\item[(b)] purely discrete, i.e. it consists only of isolated eigenvalues of finite multiplicity with no accumulation points apart from, possibly, infinity.
\end{itemize}
\end{thm} 

Part (a) has been partially proved in \manu, where we showed that all the \emph{eigenvalues} (if they exist) are real. 
The rest of Theorem \ref{thm:1} follows from
\begin{thm}\label{thm:2}
The resolvent  $(\Lp-\lambda)^{-1}$ is a compact operator on $L^2(-\pi,\pi)$ if $\lambda$ is not an eigenvalue of $\Lp$.
\end{thm} 

\begin{remark} The spectrum is always non-empty, as zero is an eigenvalue corresponding to a constant eigenfunction.
\end{remark}

In order to prove Theorem \ref{thm:2}, we show that when
$\lambda$ is not an eigenvalue of $\Lp$ (which is guaranteed, for instance, if $\lambda$ is not real)
then  the boundary
value problem 
\begin{equation} \label{e.1}
    i \epsilon (f(x)u'(x))'+iu'(x)-\lambda u(x)=F(x) \qquad
    -\pi<x<\pi 
\end{equation} 
with periodic boundary conditions $u(-\pi)=u(\pi)$ has a unique solution $u\in\mathcal{D}$ for every $F\in L^2(-\pi,\pi)$. 

The compactness of the resolvent is demonstrated by an ``explicit'' construction of a bounded Green function $G(x,s)$ such that $u(x)=\int_{-\pi}^\pi G(x,s)F(s)\,\drm s$.  The properties of $G$ are established by studying the solutions of an associated homogeneous equation in Sections \ref{sec:Frob} and \ref{sec:Green}.

\subsection*{Motivation and scope of the present paper}

Our interest in the operator \eqref{e.0} with domain determined by condition \eqref{eq:f_sym} is primarily motivated by \cite{Benilov}
and \cite{Davies}, where $f(x)=(2/\pi)\sin  x$, and  therefore \eqref{e.1} takes the form 
\be i\tilde\epsilon \left(\sin(x) u'(x)\right)' + i u'(x) -
\lambda u(x)=F(x), 
\label{eq:1sin}
\ee
with $0<\tilde \epsilon<2$. This equation arises in fluid dynamics, and describes small oscillations of a thin layer of fluid inside a rotating cylinder.
From a purely theoretical perspective, the eigenvalue problem 
associated to $L_\mathrm{per}$ has recently drawn a substantial amount of attention: see \cite{ChuPel}, \cite{Weir}, \cite{Weir2}, \cite{Davies-Weir} and \cite{ChuEtal}. Despite the fact that \eqref{eq:1sin} is highly non-self-adjoint, the spectrum of \eqref{eq:1sin} consists exclusively of real eigenvalues of finite multiplicity, it is symmetric with respect to the origin and it accumulates at $\pm \infty$. It is also known that the eigenfunctions do not form an unconditional basis of $L^2(-\pi,\pi)$.  Moreover, Davies and Weir \cite{Davies-Weir} have recently found explicit asymptotics of the eigenvalues as $\tilde{\epsilon}\to 0$.

In \cite{BLM} we established that the eigenvalues of $L_\text{per}$ are all real 
and form a symmetric set with respect to the origin. Theorem~\ref{thm:1} above
rules out completely the possibility of a non-empty essential spectrum and it
answers an open question posed in our previous paper. 

\section{A na\"{\i}ve Frobenius analysis}\label{sec:Frob}

Let $p$ satisfy the integrating factor equation
\be
     \frac{p'}{p}=\frac{f'}{f}+\frac{1}{\epsilon f}.
\label{eq:pdef}\ee
Then $u(x)$ is a solution of \eqref{e.1} iff
\begin{equation} \label{e.2}
     (p(x)u'(x))'+\frac{i\lambda p}{\epsilon f}u(x) = -\left(\frac{ip}{\epsilon f}F\right)(x)\qquad -\pi<x<\pi.
\end{equation}
For future use, we also recall the homogeneous differential equation
\begin{equation} \label{e.2h}\tag{\ref{e.2}$'$}
     (p(x)u'(x))'+\frac{i\lambda p}{\epsilon f}u(x) =0\qquad -\pi<x<\pi
\end{equation}
In order to understand how $p$ behaves near $x=0$ and $x=\pi$ it is useful to consider a simple model.
Suppose that near $x=0$, the function $f$ satisfies $f(x) = 2x/\pi$ -- this maintains the normalization $f'(0)=2/\pi$.
Then (\ref{eq:pdef}) yields $\log p = \log(2x/\pi) + \pi/(2\epsilon)\log(x)$, whence $p(x) = Cx^{1+c/\epsilon}$ where
$C$ is an arbitrary non-zero constant and $c=\pi/2$. Similarly, near $x=\pi$, we can consider a simple
model $f(x) = (2/\pi)(\pi-x)$ and obtain $p(x) = \tilde{C}(\pi-x)^{1-c/\epsilon}$. In \manu\ we proved that, under the minimal assumptions (\ref{eq:f_sym}) on $f$, the following results hold:
\begin{equation}\label{eq:p_behaviour}
    p(x)\sim \left\{\begin{array}{ll}
              x^{1+c/\epsilon} & x\sim 0 \\
              (\pi-x)^{1-c/\epsilon} & x\sim \pi, \end{array} \right.       \qquad
              \text{and} \qquad
    \frac{p(x)}{f(x)}\sim \left\{\begin{array}{ll}
              x^{c/\epsilon} & x\sim 0 \\
              (\pi-x)^{-c/\epsilon} & x\sim \pi. \end{array} \right.               
\end{equation}

By considering the model $f(x) = 2x/\pi$, $p(x) = x^{1+c/\epsilon}$ in a neighbourhood of the origin and looking for solutions of the
differential equation (\ref{e.2h}) in the form $u(x) = x^{\nu}(1+a_1x+a_2x^2+\cdots)$ one can establish the asymptotic behaviour
of solutions for this model in a neighbourhood of the origin; similarly near $x=\pi$. 
In \manu\ we show that, under hypotheses (\ref{eq:f_sym}) on $f$,  there exist solutions $u$ of  \eqref{e.2h} such that
\[
    u(x,\lambda)\sim \left\{\begin{array}{ll}
              x^{-c/\epsilon} \text{ or }1 & x\sim 0 \\
              1\text{ or }(\pi-x)^{c/\epsilon} & x\sim \pi \\
              1\text{ or }(x+\pi)^{c/\epsilon} & x \sim -\pi\end{array} \right.  
\]
This implies the existence of a unique solution $\phi(x)=\phi(x,\lambda)$ of  \eqref{e.2h}, such that
\begin{equation}\label{eq:phi_behaviour}
\phi(x,\lambda) \sim 1, \;\;\; x\rightarrow 0.
\end{equation}
The following result is crucial for reducing the problem to the interval $(0,\pi)$; in a sense it plays the 
same role in the analysis as the orthogonal splitting of the infinite matrix operator in Davies \cite{Davies} in
$\ell^2(\mathbb{Z})$ into three operators, hence reducing the problem to a problem in $\ell^2(\mathbb{N})$.
\begin{lem}
The solution $\phi$ has the symmetry property
\begin{equation} \label{eq:m1} 
\phi(-x,\lambda) = \phi(x,-\lambda). 
\end{equation}
\end{lem}
\begin{proof}
Define a function $v(x) = \phi(-x,\lambda)$. A direct calculation shows, thanks to the symmetry conditions
(\ref{eq:f_sym}), that $v$ satisfies (\ref{e.2h}) but with $\lambda$ on the right hand side replaced by 
$-\lambda$. It also satisfies $v(0)=1$. However eqn. (\ref{e.2h}) has only one solution with this property,
namely $\phi(x,-\lambda)$. This proves the result.
\end{proof}

We emphasize the important fact that  $\lambda$ is an eigenvalue of \eqref{e.2h} if and only if $\phi$ possesses the additional symmetry property
\begin{equation}\label{eq:ev_cond}
\phi(-\pi,\lambda) = \phi(\pi,\lambda).
\end{equation}

In \manu\ we also show that there is also a second solution 
$\psi(x)=\psi(x,\lambda)$ of  \eqref{e.2h} satisfying
\begin{equation}\label{eq:psi_behaviour}
\psi(x,\lambda)\sim \left\{\begin{array}{ll}
              |x|^{-c/\epsilon} & x\sim 0 \\
              | x\mp\pi |^{c/\epsilon} & x\sim \pm \pi\end{array} \right.         
\end{equation}
Observe that $\psi(-\pi, \lambda)=\psi(\pi,  \lambda)=0$, and that $\psi(x, \lambda)$ blows up when $x\sim 0$, at least when $\lambda$
is not an eigenvalue.

Consequently, when $\lambda$
is not an eigenvalue, we can also normalize $\psi(x,\lambda)$ by the condition
\[
      p(x)\psi'(x)\phi(x)-p(x)\phi'(x)\psi(x)=1   \qquad -\pi<x<\pi.
\]
The Wronskian in the  right-hand side here is obviously a constant, and below we will always assume that $\psi(x,\lambda)$ satisfies this condition.

\section{Construction of the Green function}\label{sec:Green}

Now, we are back to constructing explicitly the $L^2$ solution $u(x,\lambda)$ of \eqref{e.2} assuming that \eqref{eq:ev_cond} fails.
By the standard variation of parameters technique the general solution of \eqref{e.2} takes the form
\begin{equation}\label{eq:gen_sol}
\begin{split}
u(x,\lambda) &= \psi(x,\lambda)\int_0^x \phi(s,\lambda)\left(\frac{-\irm p}{\epsilon f}F\right)(s)\,\drm s\\
&+\phi(x,\lambda)\int_x^\pi \psi(s,\lambda)\left(\frac{-\irm p}{\epsilon f}F\right)(s)\,\drm s\\
&+A\phi(x)+B\psi(x)\,,
\end{split}
\end{equation}
where $A$ and $B$ are arbitrary complex constants.

It remains to check that one can choose constants $A$ and $B$ in such a way that 
\begin{enumerate}
\item[(a)] $u\in L^2(-\pi,\pi)$,
\item[(b)] $\Lp u\in L^2(-\pi,\pi)$, and 
\item[(c)] $u(-\pi, \lambda)=u(\pi, \lambda)$. 
\end{enumerate}
It is in fact sufficient, and easier, to check that  one can choose constants $A$ and $B$ such that 
\begin{enumerate}
\item[(a')] $u(x,\lambda)$ is continuous at $x=0$,
\item[(b')] $\irm\epsilon f u'+\irm u$ is continuous at $x=0$,
\end{enumerate}
and (c) all hold. Indeed, \eqref{eq:gen_sol} together with (a') implies (a), and together with (b') and (c) implies (b).

Note first, that by (a') and the behaviour of $\psi$ near the origin, one is tempted to take $B=0$  in \eqref{eq:gen_sol}. We shall show that this choice is indeed the right one by the careful analysis of the remaining terms in  \eqref{eq:gen_sol}. 

By the Cauchy-Schwarz inequality and \eqref{eq:p_behaviour},  \eqref{eq:phi_behaviour},  \eqref{eq:psi_behaviour} we have,
\begin{equation}\label{eq:first_integral}
\begin{split}
&\left|\psi(x,\lambda)\int_0^x \phi(s,\lambda)\left(\frac{-\irm p}{\epsilon f}F\right)(s)\,\drm s\right|\\
&\quad\le |\psi(x,\lambda)|\,\left(\int_0^x  |\phi(s,\lambda)|^2 \left|\frac{-\irm p}{\epsilon f}\right|^2\,\drm s\right)^{1/2}
\left(\int_0^x |F(s)|^2\,\drm s\right)^{1/2}\\
&\quad\le C |x|^{-c/\epsilon} |x|^{c/\epsilon+1/2}\left(\int_0^x |F(s)|^2\,\drm s\right)^{1/2}\\
&\quad\le C|x|^{1/2}\|F\|\,.
\end{split}
\end{equation}
Here, and  throughout the rest of this paper, $C$ denotes a generic positive
constant; $\| \cdot \|$ is the standard norm in $L^2(-\pi,\pi)$.

Similarly, 
\begin{equation}\label{eq:second_integral}
\begin{split}
&\left|\phi(x,\lambda)\int_x^\pi \psi(s,\lambda)\left(\frac{-\irm p}{\epsilon f}F\right)(s)\,\drm s\right|\\
&\quad\le |\phi(x,\lambda)|\,\left(\int_x^\pi |\psi(s,\lambda)|^2 \left|\frac{-\irm p}{\epsilon f}\right|^2\,\drm s\right)^{1/2}
\left(\int_x^\pi |F(s)|^2\,\drm s\right)^{1/2}\\
&\quad\le C \cdot 1\cdot |\pi-x|^{1/2}\left(\int_x^\pi |F(s)|^2\,\drm s\right)^{1/2}\\
&\quad\le C|\pi-x|^{1/2}\|F\|\,,
\end{split}
\end{equation}
so it is bounded. It is also continuous as the integrand is a product of a bounded function $\psi p/f$ and an $L^2$ function $F$.

Also, $\phi$ is continuous at zero, and $\psi$ is not, and therefore (a') holds if and only if $B=0$. 

To check the condition (b'), it is now sufficient to verify that $fu'$ is continuous at zero. Differentiating  \eqref{eq:gen_sol} with respect to $x$ gives
\begin{equation}\label{eq:gen_sol_der}
\begin{split}
fu'(x,\lambda) &= f\psi'(x,\lambda)\int_0^x \phi(s,\lambda)\left(\frac{-\irm p}{\epsilon f}F\right)(s)\,\drm s\\
&+f\phi'(x,\lambda)\int_x^\pi \psi(s,\lambda)\left(\frac{-\irm p}{\epsilon f}F\right)(s)\,\drm s\\
&+Af\phi'(x)\,,
\end{split}
\end{equation}
as the contributions from differentiating the integrals cancel out.

The last two terms go to zero as $x\to 0$, and we only need to check the continuity of the first term. Similarly to 
\eqref{eq:first_integral}, we get
\begin{equation}\label{eq:first_integral_der}
\begin{split}
&\left|f(x)\psi'(x,\lambda)\int_0^x \phi(s,\lambda)\left(\frac{-\irm p}{\epsilon f}F\right)(s)\,\drm s\right|\\
&\quad\le |f(x)\psi'(x,\lambda)|\,\left(\int_0^x  |\phi(s,\lambda)|^2 \left|\frac{-\irm p}{\epsilon f}\right|^2\,\drm s\right)^{1/2}
\left(\int_0^x |F(s)|^2\,\drm s\right)^{1/2}\\
&\quad\le C |x| |x|^{-c/\epsilon-1}|x|^{c/\epsilon+1/2}\left(\int_0^x |F(s)|^2\,\drm s\right)^{1/2}\\
&\quad\le C|x|^{1/2}\|F\|\,,
\end{split}
\end{equation}
which proves (b').

Finally, we need to guarantee that we can choose a value of constant $A$ to ensure that condition (c) holds. Direct substitution, again taking account of \eqref{eq:p_behaviour},  \eqref{eq:phi_behaviour},  \eqref{eq:psi_behaviour} gives
\begin{align*}
u(\pi,\lambda)&=A\phi(\pi,\lambda)\,,\\
u(-\pi,\lambda)&=A\phi(-\pi,\lambda)+\phi(-\pi,\lambda)\int_{-\pi}^\pi \psi(s,\lambda)\left(\frac{-\irm p}{\epsilon f}F\right)(s)\,\drm s,
\end{align*}
and as $\lambda$ is assumed not to be an eigenvalue, and therefore \eqref{eq:ev_cond} is not satisfied, we can choose
\begin{equation}\label{eq:A}
A=\int_{-\pi}^\pi \psi(s,\lambda)\left(\frac{-\irm p}{\epsilon f}F\right)(s)\,\drm s \left/ \left(\displaystyle\frac{\phi(\pi,\lambda)}{\phi(-\pi,\lambda)}-1\right)\right.\,.
\end{equation}

This  proves the existence of the resolvent of $(\Lp-\lambda)^{-1}: L_2(-\pi,\pi)\to L_2(-\pi,\pi)$ for each $\lambda$ which is not an eigenvalue of $\Lp$. Thus, the spectrum of $\Lp$ is pure point and real.

Now we proceed to writing down the expression for the Green function $G(x,s)$ of $\Lp$. We note that \eqref{eq:gen_sol} can be written, with account of $B=0$ and \eqref{eq:A}, as
\[
u(x,\lambda) = \int_{-\pi}^\pi G(x,s) F(s)\,\drm s=\int_{-\pi}^\pi (G_\mathrm{I}(x,s)+G_\mathrm{II}(x,s)+G_\mathrm{III}(x,s))F(s)\,\drm s
\]
where we set $G(x,s) :=G_\mathrm{I}(x,s)+G_\mathrm{II}(x,s)+G_\mathrm{III}(x,s)$\,, with
\begin{align*}
G_\mathrm{I}(x,s) &:= \begin{cases} 
\psi(x,\lambda) \phi(s,\lambda)\left(\frac{-\irm p(s)}{\epsilon f(s)}\right)\quad&\text{if }|x|\ge|s|\,,\\
0&\text{otherwise};
\end{cases}\\
G_\mathrm{II}(x,s) &:= \begin{cases} 
\phi(x,\lambda) \psi(s,\lambda)\left(\frac{-\irm p(s)}{\epsilon f(s)}\right)\quad&\text{if } x\le s\,,\\
0&\text{otherwise};
\end{cases}\\
G_\mathrm{III}(x,s) &:= \phi(x,\lambda) \psi(s,\lambda)\left(\frac{-\irm p(s)}{\epsilon f(s)}\right) 
/ \left(\displaystyle\frac{\phi(\pi,\lambda)}{\phi(-\pi,\lambda)}-1\right)\quad\text{for all }-\pi\le x,s\le \pi\,.\\
\end{align*}

The functions $G_\mathrm{II}(x,s)$ and $G_\mathrm{III}(x,s)$ are bounded by \eqref{eq:p_behaviour},  \eqref{eq:phi_behaviour}, \eqref{eq:psi_behaviour}, so we need to look at $G_\mathrm{I}(x,s)$. The only scope for trouble in the expression for 
$G_{\mathrm I}$ lies in the fact that $\psi(x,\lambda)$ blows up as $x\rightarrow 0$. However if $x$ is small then, in
the region $|x|\geq |s|$ where $G_{\mathrm I}$ is nonzero, \eqref{eq:psi_behaviour} and \eqref{eq:p_behaviour} yield
\[ \left|\psi(x,\lambda) \frac{p(s)}{f(s)}\right| \leq C |x|^{-c/\epsilon} |s|^{c/\epsilon} \leq C. \]

Thus $G$ is bounded and hence is the kernel of a compact operator on $L^2(-\pi,\pi)$. Thus $\Lp$ has compact resolvent and purely discrete real spectrum.

\section{Schatten class properties of the Green function}

We first recall the standard notion of Schatten class operator.
Let $T$ be a compact operator and consider its ($\infty$-dimensional)
singular value decomposition (see e.g. \cite{Gohberg-Krein}):
\[
     T=\sum _{j=0}^\infty \alpha_j|v_j\rangle \langle w_j|
\]    
where the singular values $\alpha_j\geq 0$ and the two sets of vector $\{v_j\}$ and $\{w_j\}$ are orthonormal and not necessarily equal.
Here we use the bra-ket notation: $|v\rangle \langle w| u=\langle u,w\rangle v$.
 For $p>0$, we say that $T$ is in the $p$-Schatten class, $T\in\mathcal{C}_p$, if
\[
      \|T\|_p:= \left(\sum_{j=0}^\infty \alpha_j^p\right)^{1/p}<\infty.
\]
Note that $\mathcal{C}_1$ are the trace class operators and $\mathcal{C}_2$
are the Hilbert-Schmidt operators. 

\begin{thm}
The resolvent $(\lambda-L_\text{per})^{-1}$ is in $\mathcal{C}_p$ for all
$p>1$.
\end{thm}
\begin{proof}
Let $G_{\mathrm{I}}(x,s), G_{\mathrm{II}}(x,s), G_{\mathrm{III}}(x,s)$ be the components of the Green function $G(x,s)$. It suffices to show that each of the corresponding integral operators $R_{\mathrm{I}}(\lambda), R_{\mathrm{II}}(\lambda), R_{\mathrm{III}}(\lambda)$ is in $\mathcal{C}_p$.

We start by showing that $R_{\mathrm{III}}(\lambda)\in \mathcal{C}_1$.
Indeed, $G_\mathrm{III}(x,s)=v(x)w(s)$ where 
\[
v(x) = \phi(x,\lambda)  \quad \text{and} \quad
w(s) = \psi(s,\lambda)\left(\frac{-\irm p(s)}{\epsilon f(s)}\right) 
/ \left(\displaystyle\frac{\phi(\pi,\lambda)}{\phi(-\pi,\lambda)}-1\right).
\]
By virtue of \eqref{eq:phi_behaviour}, \eqref{eq:psi_behaviour}
and \eqref{eq:p_behaviour}, both $v,w \in L^2(-\pi,\pi)$.
Then $R_\mathrm{III}(\lambda)=|v\rangle \langle w|$ and $\|R_\mathrm{III}(\lambda)\|_1=\|v\|_{L^2(-\pi,\pi)} \|w\|_{L^2(-\pi,\pi)}<\infty$, as required.

Let us now show that $R_{\mathrm{II}}(\lambda)\in \mathcal{C}_p$ for all $p>1$.
Let  
$\Omega_{\mathrm{II}}=\{(x,s)\in [-\pi,\pi]^2\,:\, x\leq s\}$, be the supports of $G_{\mathrm{II}}(x,s)$. Decompose the characteristic function of $\Omega_\mathrm{II}$ as
\[
    \mathds{1}_{\Omega_{\mathrm{II}}}(x,s)=\sum_{j=0}^\infty \sum_{i=0}^{2^j-1}
    \mathds{1}_{I_{2i,j}}(x)\mathds{1}_{I_{2i+1,j}}(s)
\]
where there are $2^{j+1}$ intervals
\[
    I_{i,j}=2\pi\left[\frac{i}{2^{j+1}},\frac{i+1}{2^{j+1}}\right]-\pi=
    \left[\frac{i\pi}{2^{j+1}}-\pi,\frac{(i+1)\pi}{2^{j+1}}-\pi \right],
    \quad i=0,\ldots,2^{j+1}-1.
\]
Then
\[
    G_\mathrm{II}(x,s)=\mathds{1}_{\Omega_\mathrm{II}}(x,s) v(x)w(s)=
    \sum_{j=0}^\infty \sum_{i=0}^{2^j-1}
    v(x)\mathds{1}_{I_{2i,j}}(x)w(s)\mathds{1}_{I_{2i+1,j}}(s).
\]
Let $S_{i,j}=| v\mathds{1}_{I_{2i,j}}\rangle \langle w\mathds{1}_{I_{2i+1,j}} |$.
Then
\[
\alpha_{i,j}:=\|S_{i,j}\|=\|v\|_{L^2(I_{2i,j})} \|w\|_{L^2(I_{2i+1,j})} 
\leq \frac{m}{2^j} 
\]
where $m$ is independent of $i$ and $j$. The constant $m$ depends on $\lambda$ and it is finite as a consequence of \eqref{eq:phi_behaviour}, \eqref{eq:psi_behaviour}
and \eqref{eq:p_behaviour}. Let 
\begin{equation} \label{e.svd_Sj}
    S_j=\sum_{i=0}^{2^{j}-1}S_{i,j}=
     \sum_{i=0}^{2^{j}-1} \alpha_{i,j} \left| \frac{v\mathds{1}_{I_{2i,j}}}{\|v\|_{L^2(I_{2i,j})}}
     \right\rangle \left\langle \frac{w\mathds{1}_{I_{2i+1,j}}}{\|w\|_{L^2(I_{2i+1,j})}} \right|.
\end{equation}
Since the intervals $I_{i,j}$ are pairwise
disjoint for a fixed $j$, the right side of \eqref{e.svd_Sj} is a singular value decomposition for $S_j$. Then
\[
    \|S_j\|_p=\left( \sum_{i=0}^{2^j-1} \alpha_i^p   \right)^{1/p}\leq
    \frac{m}{\left(2^{\frac{p-1}{p}}\right)^j}.
\]
By the triangle inequality, this ensures that
\[
    \|R_\mathrm{II}(\lambda)\|_p<\infty
\]
for all $p>1$ as required.

The fact that $R_{\mathrm{I}}(\lambda)\in \mathcal{C}_p$ for all $p>1$
follows by an analogous decomposition of the support
$\Omega_\mathrm{I}=\{(x,s)\in [-\pi,\pi]^2\,:\, |x|\geq|s|\}$
as the union of disjoint rectangles and a very similar argument.

\end{proof}

\begin{remark}
Let $\lambda_n$ be the eigenvalues of $L_\mathrm{per}$ and $\lambda\not=\lambda_n$.
By virtue of \cite[cor.XI.9.7]{Dunford-Schwartz}, the series $\sum_{n=0}^\infty (\lambda-\lambda_n)^{-p}$ converges absolutely and
\begin{equation} \label{ine}
     \sum_{j=0}^\infty |\lambda-\lambda_n|^{-p} \leq \|(\lambda-L_\mathrm{per})^{-1}\|_p^p
\end{equation}
for all $p>1$. According to the results of \cite{Weir2} on the case $f(x)=(2/\pi)\sin  x$, it is known that $\lambda_n\sim n^2$ as $n\to \infty$. Hence we know that
$(\lambda-L_{\mathrm{per}})^{-1}\not\in \mathcal{C}_{1/2}$. As $L_{\mathrm{per}}\not=L_{\mathrm{per}}^\ast$, the inequality in \eqref{ine} can not generally be reverse for any $p>0$. The question of whether $(\lambda-L_{\mathrm{per}})^{-1}\in\mathcal{C}_p$ for $p>1/2$ will be addressed in subsequent work.
\end{remark}

\end{document}